\documentclass[english]{ourlematema}
\usepackage{amsmath, amsthm, amssymb, hyperref, color}
\usepackage{MnSymbol} 
\usepackage{graphicx}
\usepackage{caption}
\usepackage[all]{xypic}
\usepackage{verbatim}
\usepackage{chemfig, chemnum}
\usepackage{tikz}
\usepackage[linesnumbered,lined,commentsnumbered]{algorithm2e}
\usepackage{caption}
\usepackage{subcaption}
\usepackage{float}
\usepackage{makecell}
\usepackage{array}
\usepackage{fancyvrb}
\usepackage{enumerate}

\newcommand{\RR}{\mathbb{R}}

\newcommand{\CC}{\mathbb{C}}

\newcommand{\cM}{\mathcal{M}}

\DeclareMathOperator{\codim}{codim}
\DeclareMathOperator{\cone}{Cone}
\DeclareMathOperator{\Star}{star}
\DeclareMathOperator{\Conv}{Conv}

\DeclareMathOperator{\Newt}{Newt}
\DeclareMathOperator{\trop}{trop}

\DeclareMathOperator{\lk}{lk}

\DeclareMathOperator{\Spec}{Spec}

 \title{Binary geometries from pellytopes}
 
\author{Lara Bossinger}
\address{%
Intituto de Matemáticas, Unidad Oaxaca,\\
Universidad Nacional Autónoma de México \\
\email{lara@im.unam.mx}
}

\author{Máté L. Telek}
\address{%
  MPI for Mathematics in the Sciences, Leipzig \\
\email{mate.telek@mis.mpg.de }
}
  \author{Hannah Tillmann-Morris}
  \address{%
 MPI for Mathematics in the Sciences, Leipzig\\
\email{hannah.tillmann-morris15@imperial.ac.uk}
}

 \date{2020/10/11}

\begin{document}

\maketitle

\begin{abstract}
\noindent
    Binary geometries have recently been introduced in particle physics in connection with stringy integrals. 
    In this work, we study a class of simple polytopes, called \emph{pellytopes}, whose number of vertices are given by Pell's numbers.
    We provide a new family of binary geometries determined by pellytopes as conjectured by He--Li--Raman--Zhang. 
    We relate this family to the moduli space of curves by comparing the pellytope to the ABHY associahedron.
\end{abstract}

\section{Introduction}

Binary geometries are affine varieties with stratifications determined by certain simplicial complexes.
These curious geometric objects -- like positive geometries -- first arose from the study of canonical forms on polytopes and amplituhedra, as a novel method to compute scattering amplitudes \cite{Arkani-Hamed:2013jha,PosGeom2017}.
The stratification of the binary geometry leads to a factorization of the amplitude.
 All systematically studied examples of binary geometries arise from simplicial complexes associated with finite type cluster algebras and generalized permutahedra~\cite{AHLT,Postnikov_permutahedra}, for instance the ABHY kinematic associahedron~\cite{AHL::Stringy}.
In this paper we initiate the study of binary geometries more abstractly by focusing on an elementary example not belonging to either class.

More precisely, given a flag simplicial complex $\Delta$ on $[n]$ we write $i\not\sim j$ for $i,j\in \Delta$ if $\{i,j\}\not \in \Delta$.
We associate to each $i\in[n]$ a polynomial in $\mathbb C[u_1,\dots,u_n]$ determining a \emph{$u$-equation} 
\begin{eqnarray}\label{eq:u-equations}
    R_i=u_i+\prod_{j\not\sim i}u_j^{a_{ij}}-1=0
\end{eqnarray}
for some integers $a_{ij}>0$. 
The affine variety in $\mathbb C^n$ defined by the $n$ equations of this form is a binary geometry if it satisfies certain boundary conditions (see Definition~\ref{Def:BinaryGeom}). In particular, a binary geometry is stratified by binary geometries corresponding to links of $\Delta$, see Lemma~\ref{Lemma:Link}.
{Notice that requiring all coordinates to be real and non-negative forces them to be in the interval $[0,1]$; hence the name \emph{binary}.} 

In this paper we focus on the simplicial complex arising from a polytope not belonging to either of the aforementioned classes: the \emph{pellytope} is defined as
\begin{eqnarray}\label{eq:pellytope}
    \mathcal P_d:=\Newt\left(\prod_{i=1}^d(1+y_i)\prod_{j=1}^{d-1}(1+y_j+y_jy_{j+1})\right) \subset \mathbb R^{d}.
\end{eqnarray}
It is a $d$-dimensional simple polytope with $3d-1$ facets and its number of vertices is given by \emph{Pell's number} $n_{d+1}$, defined recursively by 
\begin{equation}\label{eq:Pellnumber}
n_1=1,\quad n_2=2\quad \text{and} \quad n_d = 2n_{d-1} +n_{d-2}
\end{equation} 
(see Corollary~\ref{Cor:PellVertex}).
The (inner) normal fan of $\mathcal{P}_d$, denoted  $\Sigma_{d}$, determines a flag simplicial complex.
The pellytope has been studied by physicists in \cite[\S4]{HeLi::Stringy} as an example of a simplicial fan with desirable combinatorial properties: stars of $\Sigma_d$ factor as products of copies of $\Sigma_i$ for $i<d$ (see \eqref{eq:star} and Lemma~\ref{lem:stars}).
These lead the authors in \emph{loc.cit.} to conjecture that the pellytope determines a binary geometry. We verify their conjecture:

\begin{thm}\label{thm:main}
    The pellytope $\mathcal{P}_{d}$ determines a binary geometry $\widetilde{\mathcal U}_{d}\subset \mathbb C^{3d-1}$ defined by $3d-1$ $u$-equations.
\end{thm}

We call $\widetilde{\mathcal U}_{d}$ the \emph{Pellspace}.
Although the pellytope does not belong to the class of generalized permutahedra it is closely related to the associahedron.
The ABHY associahedron $\mathcal A_d$ is realized as the Newton polytope of a polynomial divisible by the polynomial defining $\mathcal P_d$ in~\eqref{eq:pellytope}, so the normal fan of $\mathcal A_{d}$ is a refinement of $\Sigma_d$.
The binary geometry associated to $\mathcal A_{d}$ is an affine chart of the moduli space of stable rational curves, denoted by $\widetilde{\mathcal M}_{0,n}$ where $n=d+3$.
This enables us to prove the following statement:

\begin{cor}\label{cor:M0n and pellytope}
$\widetilde{\mathcal M}_{0,n}$ is an affine subset of a blowup of the binary geometry $\widetilde{\mathcal U}_{d}$ of the pellytope $\mathcal P_{d}$.
\end{cor}

We therefore expect that there exists an alternative compactification of $\mathcal M_{0,n}$ in which $\widetilde{\mathcal U}_d$ is an affine chart. 
We explore the case of $d=3$ in the final section.
It would be interesting to understand the moduli interpretation and combinatorial structure of this space more generally.

\medskip\noindent
{\bf Outline.} In \S\ref{Sec:BinGeom} we recall the definition of a binary geometry and the necessary concepts from polyhedral geometry. In \S\ref{Sec::Pelltope} we study the pellytope and its combinatorics. In \S\ref{Sec:Characters} we study the very affine variety determined by the pellytope, which is followed by the proof of Theorem~\ref{thm:main} in \S\ref{sec:u-eq for pellytope} and the proof of Corollary~\ref{cor:M0n and pellytope} in \S\ref{sec:curves}.

\section{Binary geometries}
\label{Sec:BinGeom}
Let $[n] = \{1,\dots,n\}$ and consider a simplicial complex $\Delta$ on $[n]$. More precisely, $\Delta$ is a non-empty collection of subsets of $[n]$ satisfying the following property:
\begin{itemize}
    \item[(i)] If $S \in \Delta$ and $S' \subset S$ then $S' \in \Delta$.
\end{itemize}
If, additionally, a simplicial complex $\Delta$ satisfies:
\begin{itemize}
    \item[(ii)] $\{k\} \in \Delta$ for each $k \in [n]$ and 
    \item[(iii)] $\{k_1,\dots,k_r\} \in \Delta$ whenever $\{k_i,k_j\} \in \Delta$ for all $1\leq i < j \leq r$,
\end{itemize}
we call $\Delta$ a \emph{flag complex}. A simplicial complex $\Delta$ is \emph{pure} if all the maximal sets in $\Delta$ (with respect to inclusion) have the same cardinality. We say that $i, j \in [n]$ are \emph{incompatible} if $\{i, j\} \notin \Delta$, and write $i \nsim j$.
Before we proceed to recall the definition of a binary geometry, we discuss two examples of flag complexes that will serve as the running examples of this section.
\begin{exa}
\label{Ex:Sec2:Run1}
    \begin{itemize}
        \item[(a)] On $[2] = \{1,2\}$, the collection of subsets $\Delta = \big\{ \{1\},\{2\}\big\}$ form a pure flag complex where $1,2$ are incompatible.
        \item[(b)] On $[4] = \{1,2,3,4\}$, the collection of subsets 
        \[\Delta = \big\{ \{1\},\{2\}, \{3\},\{4\}, \{1,3\},\{1,4\},\{2,3\},\{2,4\} \big\}\]
        is a pure flag complex. This flag complex is combinatorially isomorphic to the inner normal fan of a square (see Figure~\ref{FIG1}). 
        In this paper, we will mostly consider flag complexes that correspond to the inner normal fans of certain polytopes. 
    \end{itemize}
\end{exa}

\begin{figure}[t]
\centering
\begin{minipage}[h]{0.3\textwidth}
\centering
\includegraphics[scale=0.35]{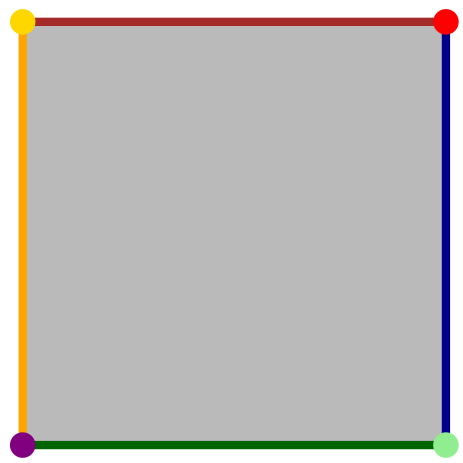}
\end{minipage}
\hspace{10 pt}
\begin{minipage}[h]{0.3\textwidth}
\centering
\includegraphics[scale=0.35]{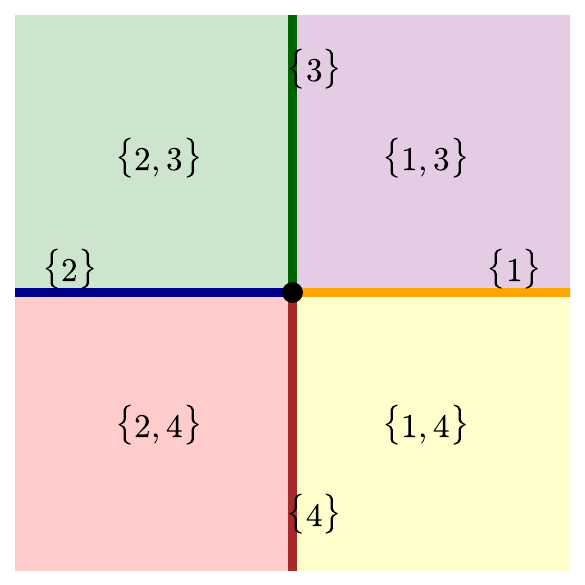}
\end{minipage}

\caption{{\small  A square and its inner normal fan, which is isomorphic to the flag complex in Example~\ref{Ex:Sec2:Run1}(b).}}\label{FIG1}
\end{figure}

\begin{dfn}
\label{Def:BinaryGeom}
\cite[Definition 2.1]{Lam::LectureNotes}
    Let $\Delta$ be a flag simplicial complex on $[n]$. A \emph{binary geometry} $\widetilde{U}$ for $\Delta$ is an  affine algebraic variety $\widetilde{U} \subset \mathbb{C}^n$ cut out by $n$ equations of the form~\eqref{eq:u-equations},
satisfying the following properties:
\begin{enumerate}[(i)]
    \item $\dim \widetilde{U} =  \max_{S \in \Delta} \# S$;
    \item for $S \subset [n]$ the subvariety $\widetilde{U}_S = \widetilde{U} \cap \{ u \in \mathbb{C}^n \mid u_i = 0 \; \forall i \in S\}$ is non-empty if and only if $S \in \Delta$;
    \item if $\widetilde{U}_S$ is non-empty, it is irreducible of codimension $\# S$ in $\widetilde{U}$.
\end{enumerate}
\end{dfn}

\begin{exa}
\label{Ex:d1BinaryGeom}
    The flag complex from Example~\ref{Ex:Sec2:Run1}(a) gives rise to the equations
    \[R_1 = u_1 + u_2 - 1 = 0, \qquad R_2 = u_2 + u_1 -1 = 0.\]
    The variety $\widetilde{U} = \{ u \in \mathbb{C}^2 \mid R_1(u) = 0, \; R_2(u) = 0\}$ has dimension $1$ -- thus property~(i) in Definition~\ref{Def:BinaryGeom} is satisfied. For $S = \{1,2\}$, the subvariety 
    \[\widetilde{U}_S = \big\{ u \in \mathbb{C}^2 \mid u_1 + u_2 - 1 = 0,\, u_1 = 0,\, u_2 = 0   \big\} \]
    is empty. For $S = \{1\}$ and $S = \{2\}$, we have
    \[\widetilde{U}_{\{1\}} = \big\{ u \in \mathbb{C}^2 \mid u_1 + u_2 - 1 = 0,\, u_1 = 0   \big\} = \big\{(0,1)\big\}, \quad \widetilde{U}_{\{2\}} = \big\{(1,0)\big\},\]
    which are irreducible subvarieties of $\widetilde{U}$ of codimension $1$. Thus, $\widetilde{U}$ is a binary geometry. By \cite[Theorem 2.9]{Lam::LectureNotes}, $\widetilde{U}$ is the only one dimensional binary geometry.
\end{exa}

If $\widetilde{U}$ is a binary geometry, then the non-empty subvarieties $\widetilde{U}_S$ for $S \in \Delta$ are also binary geometries. Their underlying flag complexes can be described as follows. The \emph{link} $\lk_\Delta S$ of $S \in \Delta$ is the simplicial complex 
\begin{eqnarray}\label{eq:link}
    \lk_\Delta S := \big\{ T \in \Delta \mid S \cap T = \emptyset \text{ and } S \cup T \in \Delta \big\} 
\end{eqnarray} 
on the set $W_S = \{j \in [n] \setminus S \mid S \cup \{ j \} \in \Delta\}$. In other words, the set $W_S$ contains all $j$ that are compatible with $S$ but do not lie in $S$.

\begin{exa}
    Consider the flag complex from Example~\ref{Ex:Sec2:Run1}(b). The link of $S = \{3\}$ equals $\lk_\Delta(\{3\}) = \big\{ \{1\} , \{2\} \big\}$.
\end{exa}

If there exists a binary geometry for a flag complex $\Delta$, then $\Delta$ is pure by \cite[Corollary 2.7]{Lam::LectureNotes}. Lemma~\ref{Lemma:Link} relates $\widetilde{U}_S$ to the link $\lk_\Delta(S)$. Part (b) will be used in the proof of Theorem~\ref{thm:main}.


\begin{lemma}
\label{Lemma:Link}
    Let $\Delta$ be a pure flag complex on $[n]$, and let  $\widetilde{U}$ be an affine algebraic variety $\widetilde{U} \subset \mathbb{C}^n$ cut out by $n$ equations of the form as in \eqref{eq:u-equations}.
For $S \in \Delta $, the subvariety $\widetilde{U}_S = \widetilde{U} \cap \{ u \in \mathbb{C}^n \mid u_k = 0 \; \forall k \in S\}$ is cut out by the equations:
\begin{align*}
    u_k &= 0 \quad \text{for } k  \in S,\\
    u_j &= 1 \quad \text{for } j \notin W_S,\\
    R_i = \; u_i\; + \prod_{\ell \in W_S\colon \; \ell \nsim  i} u_\ell^{a_{i\ell}} - 1 &= 0  \quad  \text{ for } i \in W_S.
\end{align*}
Moreover, the following hold:
\begin{itemize}
    \item[(a)] If $\widetilde{U} $ is a binary geometry and $S \in \Delta$, then $\widetilde{U}_S$ is a binary geometry with underlying simplicial complex $\lk_\Delta S$.
        \item[(b)] If $\widetilde{U}$ is an irreducible variety of dimension $\max_{S \in \Delta} \# S$, and $\widetilde{U}_{\{k\}}$ is a binary geometry with underlying simplicial complex $\lk_\Delta\{k\}$ for each $k \in [n]$,
        then $\widetilde{U} $ is a binary geometry.
\end{itemize}
\end{lemma}

\begin{proof}
We prove (b); part (a) follows directly from the proof of \cite[Proposition 2.6]{Lam::LectureNotes}. Let $S \in \Delta$ and choose $k \in S$. By assumption, $\widetilde{U}_{\{k\}}$ is a binary geometry with
\[ \dim \widetilde{U}_{\{k\}} = \max_{T \in \lk_\Delta\{k\}} \# T =  \left(\max_{V \in \Delta} \# V \right)- 1 = \dim \widetilde{U} - 1. \]
For the second equality above we use that $\Delta$ is pure, which implies that there exists $V \in \Delta$ such that $k \in V$ and  $\#V = \max_{S \in \Delta} \# S$.
Since $(S \setminus \{k\}) \in \lk_\Delta(\{k\})$ and $\widetilde{U}_{\{k\}}$ is a binary geometry by our assumption, it follows that $\widetilde{U}_S = (\widetilde{U}_{\{k\}})_{S \setminus \{k\}} $ is a non-empty irreducible variety of codimension $\#S$ in $\widetilde{U}$. Therefore, $\widetilde{U}$ is a binary geometry.
\end{proof}

The product of two binary geometries is also a binary geometry. To state this result, we recall the definition of the \emph{product of two simplicial complexes}. Let $\Delta, \Delta'$ be simplicial complexes defined on $[n], \, \{n+1,\dots,n+m\}$ respectively. The product simplicial complex is defined as
\[ \Delta \times \Delta' := \big\{F \cup F' \mid F \in \Delta, \, F' \in \Delta'  \big\}.\]
Note that every $\{k \} \in \Delta$ is compatible with every $\{k' \} \in \Delta'$ in $\Delta \times \Delta'$.

\begin{prop}\cite[Proposition 2.11]{Lam::LectureNotes}
\label{Prop:ProductBinGeom}
    Let $\widetilde{U}$ and $\widetilde{U}'$ be binary geometries for flag complexes $\Delta$ and $\Delta'$ on disjoint sets. Then  $\widetilde{U} \times \widetilde{U}'$ is a binary geometry for the product simplicial complex $\Delta \times \Delta'$.
\end{prop}

\begin{exa}
  Let $\widetilde{U}$ be the one dimensional binary geometry from Example~\ref{Ex:d1BinaryGeom}. Its product is 
 \[\widetilde{U} \times \widetilde{U}   = \big\{ u \in \mathbb{C}^4  \;\big\mid \; u_1 + u_2 - 1 = 0, \; u_3 + u_4 - 1 = 0  \big\}.\]
 This variety equals the binary geometry given by the flag complex from Example~\ref{Ex:Sec2:Run1}(b). 
 For $S = \{3\}$, the subvariety $(\widetilde{U} \times \widetilde{U})_{\{3\}}$ is cut out by the equations $u_3 = 0$, $u_4 = 1$ and $u_1 + u_2 -1 = 0$. Thus $(\widetilde{U} \times \widetilde{U})_{\{3\}}$ is isomorphic to the one dimensional binary geometry $\widetilde{U}$ from Example~\ref{Ex:d1BinaryGeom}.
\end{exa}

\section{Pellytopes}
\label{Sec::Pelltope}

For each $d \in \mathbb{N}$, the pellytope $\mathcal{P}_d \subset \mathbb{R}^d$ defined in \eqref{eq:pellytope} is the Minkowski sum of the polytopes
\begin{equation}\label{Eq:PiQi}
    P_i = \Conv \{0,e_i\}, \quad    Q_j = \Conv \{0,e_j,e_j+e_{i+1}\},
\end{equation}
where $ i \in [d], j \in[d-1]$ and $e_1,\dots,e_d$ denote standard basis vectors of $\mathbb{R}^d$. 
In Corollary~\ref{Cor:PellVertex}, we show that the number of vertices of $\mathcal{P}_d$ is given by Pell's number \eqref{eq:Pellnumber}.  
For $d=1$, the pellytope is the segment $\Conv\{0,e_1\}$ and its inner normal fan is combinatorially isomorphic to the flag complex in Example~\ref{Ex:Sec2:Run1}(a). 
For $d=2$, we depicted the pellytope and its inner normal fan in Figure~\ref{FIG2}. The main result of this section is:

\begin{prop}
\label{Prop:PellyFan}
    For each $d \in \mathbb{N}$, the inner normal fan $\Sigma_d$ of the pellytope $\mathcal{P}_d$ is a simplicial fan with $n_{d+1}$ maximal cones and $3d-1$ rays spanned by the vectors $e_1, \dots , e_d,-e_1, \dots , -e_d, e_1 - e_2, \dots , e_{d-1} - e_d$.
\end{prop}

\begin{figure}[t]
\centering
\begin{minipage}[h]{0.5\textwidth}
\centering
\includegraphics[scale=0.55]{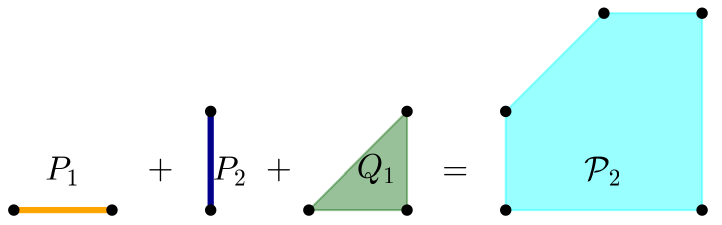}

\vspace{25 pt}
{\small (a)}
\end{minipage}
\hspace{20 pt}
\begin{minipage}[h]{0.3\textwidth}
\centering
\includegraphics[scale=0.3]{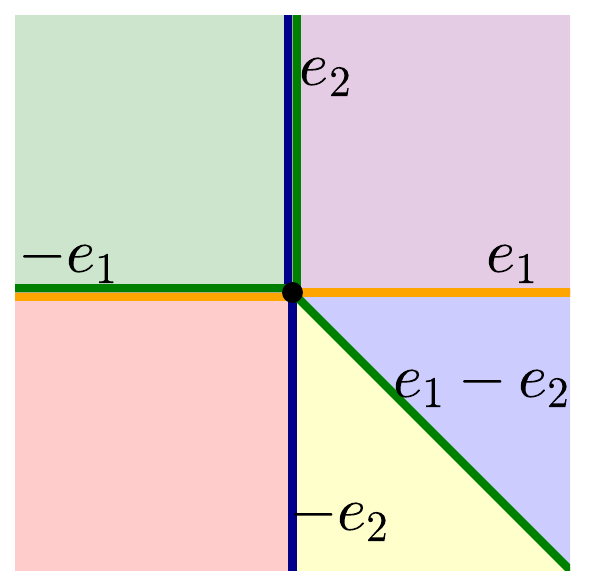}

{\small (b)}
\end{minipage}

\caption{{\small  (a) The pellytope $\mathcal{P}_2$. (b) The inner normal fan of $\mathcal{P}_2$, which is the common refinement of the inner normal fans of $P_1,\, P_2$ and $Q_1$. }}\label{FIG2}
\end{figure}

Before proving Proposition~\ref{Prop:PellyFan}, we recall some basic notions of polyhedral geometry -- for details see \cite{Ziegler:book}. 
Given a polytope $P \subset \mathbb{R}^n$ we denote its \emph{inner normal fan} by~$\Sigma_P$. 
To a fan $\Sigma$ one associates a simplicial complex as follows.
For a fixed order $\rho_1, \dots, \rho_k$ of the rays of $\Sigma$, we define
\[ 
\Delta(\Sigma) := \big\{ S \subset [k] \; \big\mid \;\cone(\rho_i \mid i \in S ) \in \Sigma \big\}. 
\]
For an example, we refer to Figure~\ref{FIG1}. If $\Sigma$ is a \emph{simplicial fan}, that is, each cone $C \in \Sigma$ is generated by linearly independent vectors, then $\Delta(\Sigma)$ is a flag complex.

Denote the {Minkowski sum} of two polytopes $P,Q \subset \mathbb{R}^n$ by $P+Q$; its inner normal fan is the \emph{common refinement} $\Sigma_P\wedge \Sigma_Q$ of $\Sigma_P$ and $\Sigma_Q$, that is,
\begin{align}
\label{Eq:MinkowskiRefinement}
\Sigma_{P+Q} = \Sigma_P \wedge \Sigma_Q = \big\{ C \cap C' \mid C \in \Sigma_P, \, C' \in \Sigma_Q\big\}. 
 \end{align}

For $P \subset \mathbb{R}^n$ (resp. $Q \subset \mathbb{R}^{m}$) denote $\iota_n (P)$ (resp. $\iota_m (Q)$) the inclusion of $P$ (resp. $Q$) into the first $n$ (resp. last $m$) coordinate hyperplanes of $\mathbb{R}^{n+m}$. A simple computation shows that 
\begin{align}\label{Eq:DirectProd}
\Sigma_{ \iota_n (P)  + \iota_m (Q)} = \Sigma_P\times \Sigma_Q =  \big\{ C \times C' \; \big \mid \; C \in  \Sigma_P, \, C' \in \Sigma_Q\big\}.
\end{align}

\begin{proof}[Proof of Proposition~\ref{Prop:PellyFan}]
  By definition, the pellytope $\mathcal{P}_d$ is the Minkowski sum $\mathcal{P}_{d-1} + P_d +  Q_{d-1}$.  From \eqref{Eq:MinkowskiRefinement} and \eqref{Eq:DirectProd} it follows that 
  \[
  \Sigma_d = \big( \Sigma_{d-1} \times \Sigma_{P_d} \big) \wedge \Sigma_{Q_{d-1}}.
  \] 
    The fan $\Sigma_{P_d}$ has two rays, generated by $e_d,-e_d$.
    The rays of $\Sigma_{Q_{d-1}}$ are generated by $e_d,-e_{d-1},e_{d-1} - e_d$. Moreover, $\Sigma_{Q_{d-1}}$  has three $2$-dimensional cones $\cone(e_d,-e_{d-1})$, $\cone(e_d,e_{d-1} - e_d)$, $ \cone(-e_{d-1},e_{d-1} - e_d)$ and a lineality space of dimension $d-2$.

Using \eqref{Eq:DirectProd}, the fan  $\Sigma_{d-1}\times \Sigma_{P_d}$ is straightforward  to compute. 
One adds to every cone in $\Sigma_{d-1}$ either a ray $\cone(e_d), \cone(-e_{d})$ or $0$. 
All rays but $\cone(e_{d-1}-e_d)$ of $\Sigma_{Q_{d-1}}$ are also rays of $\Sigma_{d-1} \times \Sigma_{P_d}$.
Thus, $\Sigma_d$ is the refinement of $\Sigma_{d-1} \times \Sigma_{P_d}$ by adding the ray generated by $e_{d-1}-e_d$.

Using the above observations, we complete the proof by induction on $d$. 
For $d=1,2$, the normal fan $\Sigma_1$ (resp. $\Sigma_2$) is simplicial and has $2$ (resp. $5$) rays and maximal cones (see Figure~\ref{FIG2}(b) for $d=2$). 
Assume that the statement of the proposition is true for $k \leq d-1$ for some $d \geq 3$. Since $\Sigma_{d-1}$ is simplicial, the fan $\Sigma_{d-1} \times \Sigma_{P_d}$ is also simplicial and has twice as many maximal cones as $\Sigma_{d-1}$. 

 The ray $\cone(e_{d-1}-e_d)$ is contained in $\cone(e_{d-1},-e_d) \in \Sigma_{d-1}\times \Sigma_{P_d}$ ,which implies that $\Sigma_d$ is a simplicial fan.   Since $\Sigma_d$ is the refinement of $\Sigma_{d-1}\times \Sigma_{P_d}$ by adding the ray generated by $e_{d-1}-e_d$, it follows that $\Sigma_d$ has $3(d-1)+2+1=3d-1$ rays.
 Every maximal cone of $\Sigma_{d-1}\times \Sigma_{P_d}$  containing $e_{d-1} -e_d$ contributes two maximal cones of $\Sigma_d$ (so one extra). 
 The number of  maximal cones of $\Sigma_{d-1}\times \Sigma_{P_d}$  containing $e_{d-1} -e_d$ is the number of $(d-1)$-dimensional cones in $\Sigma_{d-1}$ that contain $e_{d-1}$ which is the number of maximal cones in $\Sigma_{d-1}$. Thus,  $\Sigma_d$ has $2n_{d} + n_{d-1}$ maximal cones.
\end{proof}

\begin{cor}
\label{Cor:PellVertex}
    For each $d \in \mathbb{N}$, the pellytope $\mathcal{P}_d$ is a simple polytope with $3d-1$ facets and $n_{d+1}$ vertices.
\end{cor}




We finish this section with the definition of the \emph{star} of a fan. Let $\tau$ be a cone in a fan $\Sigma\subset\mathbb{R}^n$ and denote by $L(\tau)$ the linear span of $\tau$.
The \emph{star of $\tau$} is defined as the projection of the cones in $\Sigma$ containing $\tau$ under the natural projection map $\pi\,\colon \,\mathbb{R}^n \to\mathbb{R}^n/L(\tau)$, that is, 
\begin{eqnarray}\label{eq:star}
    \Star_\Sigma(\tau) := \{ \pi(\sigma) \mid \tau \subset \sigma \in \Sigma \}.
\end{eqnarray}
This construction is similar to the link complex discussed in \S\ref{Sec:BinGeom}. In fact, $\lk_{\Delta(\Sigma)}\tau$ is combinatorially isomorphic to $\Delta(\Star_\Sigma(\tau))$. 

We conclude this section by investigating the stars of rays in $\Sigma_d$. 
Combined with Lemma~\ref{Lemma:Link}, this allows us to prove that the pellytopes define binary geometries by induction on $d$.

\begin{lemma}\label{lem:stars}
Let $\rho_v$ denote the ray in $\Sigma_d$ generated by the vector $v \in \mathbb{R}^d$, then
\[
\Star_{\Sigma_d}(\rho_v)=\begin{cases}
        \Sigma_{d-1}& \text{if } v\in \{\pm e_1,\pm e_d\},\\
                \Sigma_{i-1}\times \Sigma_{d-i} & \text{if } v \in \{\pm e_i \mid i = 2, \ldots, d-1\},\\
        \Sigma_{d-1}\times \Sigma_{1} & \text{if } v\in \{e_1-e_2,e_{d-1}-e_d\},\\
        \Sigma_{i-1}\times \Sigma_{1}\times \Sigma_{d-i-1} & \text{if } v\in \{e_i-e_{i+1}\mid i=2,\dots,d-2\}.
        \end{cases}
\]
\end{lemma}

\begin{proof}
First, we consider the case $v = e_1$. Modulo the linear space $L(e_1)$ spanned by $e_1$, the rays $\rho_{e_1}$ and $\rho_{-e_1}$ become zero and $-e_2 \equiv e_1 - e_2 \mod L(e_1)$. 
It follows from this that $ \Star_{\Sigma_d}(\rho_{e_1})= \Sigma_{d-1}$. 
For $v = e_i$ where $i =2,\dots,d-1$, we have $e_i \equiv -e_i \equiv 0 \mod L(e_i)$, $ e_{i-1} - e_i \equiv e_{i-1} \mod L(e_i)$ and $e_{i} - e_{i+1} \equiv -e_{i+1} \mod L(e_i)$. Thus, $\Star_{\Sigma_d}(\rho_{e_i}) = \Sigma_{i-1}\times \Sigma_{d-i}$. The other cases follow by similar arguments.
\end{proof}





    

\section{Pellspace}\label{Sec:PellSpace}

In this section, we introduce the Pellspace, compute its character lattice, and show that it forms a binary geometry for the simplicial complex defined by the inner normal fan of the pellytope. 
Throughout this section, we closely follow the notation from \cite{Lam::LectureNotes}.

\subsection{The character lattice and bounded characters}\label{Sec:Characters}

For $i\in [d],\; j\in [d-1]$, consider the polynomials
\begin{align*}
    p_i := 1 + y_i \quad \text{and} \quad  q_j &:= 1 + y_j + y_j y_{j+1} \quad \text{in} \quad \mathbb C[y_1,\dots,y_d].
\end{align*}
The Newton polytope of $p_i$ (resp. $q_j$) equals $P_i$ (resp. $Q_j$) from \eqref{Eq:PiQi}. 
We define the \emph{open Pellspace} as the very affine variety in $(\mathbb C^*)^d$ complement to the vanishing set of the $p_i$'s and $q_j$'s:
\begin{align*}
\mathcal U_{d} := \big\{ x \in (\mathbb{C}^*)^{d} \; \big\mid \; p_i(x) \neq 0 \; \forall i\in [d] \text{ and } q_j(x) \neq 0 \; \forall j\in[d-1] \big\}
\end{align*}
Since the $p_i$'s and $q_j$'s only have positive coefficients, the positive real orthant $\mathbb{R}^d_{>0}$ is a connected component of $\mathcal{U}_d$. 
By \cite[Lemma 5.2]{Lam::LectureNotes}, the character lattice $\Lambda$ of $\mathcal{U}_d$ is  the lattice of Laurent monomials in  $y_1, \dots, y_n, \, p_1, \dots, p_d, \, q_1, \dots, q_{d-1}$. 
In particular, every element of $\Lambda$ can be written uniquely as
\begin{align}
\label{Eq:LatticeElem}
   y^a p^b q^c   = \prod_{i=1}^d y_i^{a_i} \prod_{i=1}^d p_i^{b_i} \prod_{i=1}^{d-1} q_i^{c_i}, \quad (a,b,c) \in \mathbb{Z}^d \times  \mathbb{Z}^d \times  \mathbb{Z}^{d-1}.
\end{align}
A character $y^a p^b q^c \in \Lambda $ is \emph{bounded} if it takes bounded values on $\mathbb{R}^d_{>0}$. We denote by $\Gamma \subset \Lambda$ the semigroup of bounded characters. In what follows, we show that the minimal generators of $\Gamma$ form a basis of $\Lambda$.

To find the minimal generators of $\Gamma$, we recall the method from \cite{AHL::Stringy}.
The \emph{tropicalization} of rational functions in $y_1, \dots , y_d$ is defined by
\[ y_i \mapsto Y_i, \quad + \mapsto \min, \quad \times \mapsto +, \quad \div \mapsto -.\]
For example, the tropicalization of the monomial $y^a p^b q^c$ from \eqref{Eq:LatticeElem} equals
\begin{align}
\trop( y^a p^b q^c )(Y)  
  &= \sum_{i=1}^d a_i Y_i + \sum_{i=1}^d b_i \min\{0,Y_i\} + \sum_{i=1}^{d-1} c_i \min\{0,Y_i, Y_i + Y_{i+1} \}. \nonumber 
\end{align}
By \cite[Lemma 5.16]{Lam::LectureNotes} 
we have 
\begin{equation}
\label{Eq:DescripGamma}
\begin{aligned}
 \Gamma = \big\{ y^a p^b q^c \in \Lambda \; \big\mid\; 
 \trop(y^a p^b q^c)(Y) \geq 0 \quad  \forall  Y \in \mathbb{R}^d \big\}.
\end{aligned}
\end{equation}

A simple computation shows that 
    $\trop(y^a p^b q^c)(Y) =  0$ for all $Y \in \mathbb{R}^d$  if and only if $(a,b,c) = 0$. Thus, if  $y^a p^b q^c \neq 1$, the inequality  $\trop(y^a p^b q^c)(Y_*) \geq 0$ is strict for at least one $Y_* \in \mathbb{R}^d$.
Since $\trop(y^a p^b q^c)(Y)$ is a piecewise linear function on $\mathbb{R}^d$ and linear on each cone of $\Sigma_d$, it is enough to check whether $\trop(y^a p^b q^c)(v) \geq 0$ for generators $v$ of the rays in $\Sigma_d$.

\begin{lemma}
\label{Lemma:PrimitiveEnough}
Let $V$ be a set of generators of the rays in $\Sigma_d$. The following are equivalent
    \begin{enumerate}[(i)]
        \item $\trop(y^a p^b q^c)(Y) \geq  0$ for all $Y \in \mathbb{R}^d$ and $\trop(y^a p^b q^c)(Y_*) >  0$ for at least one $Y_* \in \mathbb{R}^d$,
        \item $\trop(y^a p^b q^c)(v) \geq  0$ for all $v \in V$ and $\trop(y^a p^b q^c)(v_*) >  0$ for at least one $v_* \in V$.
    \end{enumerate}
\end{lemma}

\begin{proof} 
\textit{(i) $\Leftarrow$ (ii).} Let $Y \in \mathbb{R}^d$.
Since  $\Sigma_d$ is a complete fan, there exists a cone $C \in \Sigma_d$ (not necessarily full dimensional) with $Y \in C$. 
Let $v_1, \dots , v_{s} \in V$ be the generators of $C$ and let $\lambda_1, \dots , \lambda_{s} \geq 0$ such that $\sum_{i=1}^{s} \lambda_i v_i = Y$. Since $\trop(y^a p^b q^c)$ is linear on $C$ it 
follows that
\begin{align*}
 \trop(y^a p^b q^c)(Y) = \trop(y^a p^b q^c)\left(\sum_{i=1}^{s} \lambda_i v_i\right) = \sum_{i=1}^{s} \lambda_i  \underbrace{\trop(y^a p^b q^c)( v_i)}_{\geq 0} \geq 0.
 \end{align*}
The second part of (i) follows by taking $Y_* = v_*$.

\textit{(i) $\Rightarrow$ (ii).} Assumption (i) implies $\trop(y^a p^b q^c)(v) \geq  0$ for all $v \in V$. It therefore suffices to find one $v_* \in V$ such that $\trop(y^a p^b q^c)(v_*) >  0$. By assumption, there exists $Y_* \in \mathbb{R}^d$ such that $\trop(y^a p^b q^c)(Y_*) >  0$. Again, we write $Y_* = \sum_{i=1}^{s} \lambda_i v_i$ for some $v_1, \dots , v_s \in V$, $\lambda_1, \dots , \lambda_s \geq 0$. If $\trop(y^a p^b q^c)(v_i) =  0$ for all $i = 1, \dots ,s$, then $\trop(y^a p^b q^c)(Y_*) =  0$, which is a contradiction. Thus, there exists $v_* \in \{v_1, \dots ,v_s\}$ with $\trop(y^a p^b q^c)(v_*) >  0$.
\end{proof}


Now let $V= (e_1, \dots , e_d, -e_1, \dots -e_d, e_1 - e_2, \dots ,e_{d-1} - e_d)$ be the tuple of ray generators for $\Sigma_d$.
Consider the real matrix $M_d=(m_{ij})_{i,j\in[3d-1]}$ defined~by
\begin{equation}\label{Eq::Mmatrix}
    m_{ij} = \trop(F_i)(V_j),
\end{equation}
where $(F_1, \dots, F_{3d-1}) = (y_1, \dots ,y_d, p_1, \dots, p_d, q_1, \dots , q_{d-1})$.

\begin{exa} Consider the case $d=2$. We have
\begin{align*}
    \trop(F_5)(e_1-e_2) = \trop(1+y_1+y_1y_2)(1,-1) = \min \{ 0,1,1-1\} = 0.
\end{align*}
Thus, the entry in the bottom left corner of $M_2$ equals $0$. We compute the remaining entries of $M_2$ using the same approach and obtain
\begin{equation}
M_2=\left(\begin{smallmatrix}
    1  & 0 & -1  & 0 & 1\\0  & 1 & 0  & -1 & -1\\ 0  & 0 & -1  & 0 & 0\\ 0  & 0 & 0  & -1 & -1\\  0  & 0 & -1  & -1  & 0
\end{smallmatrix}\right) \in \mathbb{R}^{5\times 5}.   
\end{equation}
\end{exa}

Combining \eqref{Eq:DescripGamma} and Lemma~\ref{Lemma:PrimitiveEnough}, it follows that the bounded characters on $\mathcal{U}_d$ are given by
\begin{align}
\label{Eq:GammawithM}
     \Gamma &= \big\{ y^a p^b q^c \in \Lambda \;\big \mid\;  (a,b,c) M_d	\geq 0 \Big\}.
\end{align}
Here we use the notation $v \geq 0$ for $v \in \mathbb{R}^n$ to indicate that each entry of the vector $v$ is non-negative.

\begin{lemma}\label{Lemma:M invertible}
The matrix $M_d$ from~\eqref{Eq::Mmatrix} is invertible over $\mathbb{Z}$. The rows of the inverse matrix $M_d^{-1}$ are 
\begin{eqnarray*}
&\beta_i = e_i + e_{d+i+1} -e_{2d+i }, \qquad  
    \beta_d = e_{d} - e_{2d} , \qquad
     \beta_{d+1} = -e_{d+1}, &\\
&     \beta_{d+j}  = e_{d+j-1} - e_{2d+j-1} \qquad
     \beta_{2d+i} = - e_{d+i} - e_{d+i+1} + e_{2d+i}.  &
\end{eqnarray*}
for $i=1,\dots,d-1$ and $j=2,\dots,d$.
\end{lemma}

\begin{proof}
Let $A_1, \dots , A_d, B_1, \dots, B_d, C_1, \dots , C_{d-1}$ denote the columns of $M_d$. Explicitly we have $A_i=e_i$ for $i\in [d]$, 
\[
B_1 = -e_{1} - e_{d+1} - e_{2d+1}, \quad  B_i = -e_{i} - e_{d+i} - e_{2d+i-1} - e_{2d+i}, \quad B_d = -e_{d} - e_{2d} - e_{3d-1},
\]
for $i=2,\dots,d-1$, and 
\[
C_i = e_i - e_{i+1} -e_{d+i+1} - e_{2d+i+1}, \qquad \ C_{d-1} = e_{d-1} - e_{d} -e_{2d}.
\]
for $i=[d-2]$.
A direct computation shows that $M_d^{-1} M_d$ is the identity matrix.
\end{proof}
For example, when $d=2$ we have
\[
M_2^{-1}=\left(\begin{smallmatrix}
     1  & 0 & 0  & 1 & -1\\0  & 1 & 0  & -1 & 0\\ 0  & 0 & -1  & 0 & 0\\ 0  & 0 & 1  & 0 & -1\\  0  & 0 & -1  & -1  & 1
\end{smallmatrix}\right).
\]

\begin{prop}\label{Prop:Gen_ui}
Let $\beta_i = (a_i,b_i,c_i), \, i = 1, \dots 3d-1$ denote the rows of $M_d^{-1}$. 
The elements
\[ 
u_i = z^{\beta_i} =  y^{a_i} p^{b_i} q^{c_i} = \prod_{j=1}^d y_j^{a_{ij}} \prod_{j=1}^d p_j^{b_{ij}}  \prod_{j=1}^{d-1} q_j^{c_{ij}}  
\]
are the minimal generators of $\Gamma$. 
Moreover, $u_1, \dots , u_{3d-1}$ is a basis of the character lattice $\Lambda$.
\end{prop}

\begin{proof}
Since $\beta_i M_d = e_i \geq  0$, it follows from~\eqref{Eq:GammawithM} that $u_i = z^{\beta_i} \in \Gamma$. 
We show that each $u_i$ is a minimal element of $\Gamma$ by contradiction. 
Assume that there exists $w_1, w_2 \in \Gamma$ such that $u_i = w_1 w_2$. Let $\alpha_1, \alpha_2 \in \mathbb{Z}^{3d-1}$ such that $w_1 = z^{\alpha_1}, w_2 = z^{\alpha_2}$. By construction $\alpha_1 + \alpha_2 =  \beta_i$. 
After multiplying these vectors with the matrix $M_d$ from the right, we have
\begin{align}\label{Eq::ProofContr}
\alpha_1 M_d + \alpha_2 M_d = \beta_i M_d = e_i
\end{align}
Since $w_1, w_2 \in \Gamma$, we also have that $\alpha_1 M_d \geq 0$ and $\alpha_2M_d \geq 0$. 
If there exists $k \neq \ell$ such that the $k^{\text{th}}$ coordinate of $\alpha_1M_d$ and $\ell^{\text{th}}$ coordinate of $\alpha_2M_d$ are both nonzero, then $\alpha_1 M_d + \alpha_2 M_d = e_i$ has at least two non-zero coordinates, which is a contradiction. 
If there exists $k$ such that the $k^{\text{th}}$ coordinate is positive for both $\alpha_1 M_d$ and $\alpha_2 M_d$, then the $k^{\text{th}}$ coordinate of $\alpha_1 M_d + \alpha_2 M_d = e_i$ is larger than one, which is again a contradiction. 
Thus $u_i$ cannot be written as $w_1w_2$ for $w_1,w_2 \in \Gamma$.

It remains to show that $u_1, \dots , u_{3d-1}$ generate $\Gamma$ and that they form a basis of $\Lambda$. Let $w = z^\alpha \in \Lambda$ and choose $\lambda_1, \dots, \lambda_{3d-1} \in \mathbb{Z}$ such that
\[ 
\alpha M_d = \sum_{i=1}^{3d-1} \lambda_i e_i =  \sum_{i=1}^{3d-1} \lambda_i (\beta_i M_d)  =  \big(\sum_{i=1}^{3d-1} \lambda_i \beta_i \big)M_d.  
\]
Multiplying with $(M_d)^{-1}$ from the right, we have $\alpha =  \sum_{i=1}^{3d-1} \lambda_i \beta_i$, which shows that $u_1, \dots, u_{3d-1}$ form a basis of $\Lambda$. 
If we assume additionally that $w \in \Gamma$, then $\alpha M_d \geq 0$ and we can choose $\lambda_1, \dots, \lambda_{3d-1} \in \mathbb{Z}_{\geq 0}$.
From this it follows that $u_1, \dots , u_{3d-1}$ generate $\Gamma$.
\end{proof}

\begin{cor}\label{cor:ui monomials}
The minimal generators of $\Gamma$ have the following form:
\[
u_i = \tfrac{y_ip_{i+1}}{q_i}, \quad u_d = \tfrac{y_d}{p_d}, \quad  u_{d+1} = \tfrac{1}{p_1}, \quad   u_{i+d+1} = \tfrac{p_i}{q_i}, \quad u_{i+2d} = \tfrac{q_i}{p_ip_{i+1}}
\]
for $i = 1, \dots ,d-1$. 
\end{cor}

\begin{exa}
For $d=2$, the minimal generators of $\Gamma$ are
\[
u_1 = \tfrac{y_1 p_2}{q_1}, \quad u_2 = \tfrac{y_2}{p_2},\quad u_3 = \tfrac{1}{p_1}, \quad u_4 = \tfrac{p_1}{q_1},\quad u_5 = \tfrac{q_1}{p_1p_2}.
\]
\end{exa}

\subsection{The {$u$-}equations for the Pellspace}\label{sec:u-eq for pellytope}

Given the expressions in Corollary~\ref{cor:ui monomials}, we observe that the minimal generators of $\Gamma$ satisfy the following equations:
    \begin{equation}\label{eq:equations Gamma}
    u_i = \begin{cases}
        1 - u_{d+1}u_{d+2} \quad & i=1\\
        1 -  u_{i+d} u_{i+d+1} u_{i-1 + 2d} \quad & i= 2, \ldots, d-1 \\
        1 - u_{3d-1}u_{2d} \quad & i=d\\
        1 - u_1u_{1+ 2d} \quad & i=d+1\\
        1 - u_j u_{j+1} u_{j+1 +2d} \quad & i = j + d + 1 \text{ where } 1\leq j \leq d-2\\
        1 - u_{d-1}u_{d} \quad & i = 2d \\
        1 - u_{d+1}u_2u_{2+2d} \quad & i=1+2d\\
        1 - u_{j-1+2d}u_{j+d}u_{j+1}u_{j+1+2d} \quad & i= j+2d \text{ where } 2 \leq j \leq d-2\\
        1 - u_d u_{3d-2} u_{2d-1} \quad & i = 3d-1
    \end{cases}
        \end{equation}


As expected these equations are $u$-equations for the pellytope:

\begin{lemma}\label{lem:u-eqn Gamma}
    The equations \eqref{eq:equations Gamma} satisfied by the minimal generators of $\Gamma$ are $u$-equations for $\Delta(\Sigma_d)$ as defined in \eqref{eq:u-equations}.
\end{lemma}

\begin{proof}
The matrix $M_d$ induces the following correspondence between rays of $\Sigma_d$ and generators of $\Gamma$:
\begin{equation}\label{eq: ui to ray dictionary}
    u_i \longleftrightarrow \begin{cases}
           \cone(e_i) \quad & \text{if $i \in [d]$}\\
       \cone(-e_{j}) \quad & \text{if $i = d+j$ where  $j\in [d]$}\\
       \cone(e_{j} - e_{j+1}) \quad & \text{if $i = 2d+j$ where $j \in [d-1]$}
       \end{cases}
\end{equation} 
The simplices in $\Delta(\Sigma_d)$ given by the rays $\rho$ and $\rho'$ are compatible if and only if $\cone(\rho, \rho') \in \Sigma_d$. In this case we say that the rays are compatible.
Note that 
    \begin{enumerate}
        \item The rays generated by $\pm e_i, \pm e_{i+1}$ and $e_i - e_{i+1}$ are coplanar, so any pair of these five rays are only compatible if they are adjacent in the two-dimensional fan $\Sigma_d \cap (L(e_i) + L(e_{i+1}))$.
        \item  The cone spanned by $e_i$ and $-e_{i+2}$ lies in $\Sigma_d$ and bisects the cone spanned by $e_i - e_{i+1}$ and $e_{i+1} - e_{i+2}$, so the two rays $\cone(e_i - e_{i+1})$ and $\cone(e_{i+1} - e_{i+2})$ are incompatible.
        \item All other rays are compatible. This can be checked by induction, using the recursive construction of $\Sigma_d$ from $\Sigma_{d-1}$.
        The $d$-dimensional fan $\Sigma_{d-1} \times \Sigma_1e_d$ contains two rays which are not contained in $\Sigma_{d-1}$, namely $\cone(e_d)$ and $\cone(-e_d)$.
        These two rays are incompatible with each other, and compatible with every ray in $\Sigma_{d-1}$. 
        Any two rays in $\Sigma_{d-1}$ are compatible in $\Sigma_{d-1} \times\Sigma_1e_d$ if and only if they are compatible in $\Sigma_{d-1}$.
        
        Now $\Sigma_d$ is the refinement of $\Sigma_{d-1} \times \Sigma_1e_d$ by the ray $\rho := \cone(e_{d-1} - e_d)$.
        This ray bisects the two-cone spanned by $e_{d-1}$ and $-e_d$, so the two rays $\cone(e_{d-1})$ and $\cone(-e_d)$ are the only compatible pair in $\Sigma_{d-1} \times \Sigma_1e_d$ to become incompatible in $\Sigma_d$.
        Moreover, $\rho$ is compatible with a ray $\rho'$ in $\Sigma_d$ if and only if $\rho'$ is compatible with both $\cone(e_{d-1})$ and $\cone(-e_d)$ in $\Sigma_{d-1} \times \Sigma_1e_d$.
        Thus the only rays of $\Sigma_d$ which are incompatible with $\rho$ are $\cone(-e_{d-1})$, $\cone(e_d)$ and $\cone(e_{d-2}-e_{d-1})$.
    \end{enumerate}
It follows that the equations in \eqref{eq:equations Gamma} are of the desired form \eqref{eq:u-equations}. 
\end{proof}  

We note that correspondence~\eqref{eq: ui to ray dictionary} can also be observed on the level of the Newton polytopes: after multiplying each variable $u_i$ so that the denominator is of form $q:= \prod_{j \in [d]} p_j \prod_{j \in [d-1]} q_j$, we see that the Newton polytope of the enumerator is the convex hull of all lattice points in $\mathcal P_d$ that do not lie in the facet determined by the associated ray in $\Sigma_d$.

\begin{exa}
When $d=3$ the minimal generators of $\Gamma$ are 
\[
    u_1=\tfrac{y_1p_2}{q_1}, \mkern9mu
    u_2=\tfrac{y_2p_3}{q_2}, \mkern9mu  u_3=\tfrac{y_3}{p_3}, \mkern9mu  u_4=\tfrac{1}{p_1},\mkern9mu  u_5=\tfrac{p_1}{q_1}, \mkern9mu u_6=\tfrac{p_2}{q_2}, \mkern9mu  u_7=\tfrac{q_1}{p_1p_2}, \mkern9mu  u_8=\tfrac{q_2}{p_2p_3}.
\]
As indicated above they correspond to the rays $e_1,\, e_2,\, e_3,\, -e_1,\, -e_2,\,-e_3,\, e_1-e_2$ and $e_2-e_3$ respectively. For example, we have 
\[
    u_3=\frac{y_3}{p_3}=\frac{y_3p_1p_2q_1q_2}{p_1p_2p_3q_1q_2}
\]
As $y_3=p_3-1$ all lattice points in $\Newt(y_3p_1p_2q_1q_2)$ have a positive $y_3$ coordinate. Hence, it is the complement of the facet with normal vector $e_3$.
\end{exa}

\begin{dfn}\label{def:pellspace}
    The vanishing locus of the $u$-equations for $\Delta(\Sigma_d)$ in $\mathbb{C}^{3d-1}$ is called \emph{Pellspace} and it is denoted by $\widetilde{\mathcal{U}}_d$.
\end{dfn}

Theorem~\ref{thm:main} states that the Pellspace is a binary geometry with simplicial complex given by the inner normal fan of the pellytope. 
We proceed by showing that the Pellspace is irreducible by realizing its defining ideal as the kernel of a ring homomorphism.
Define $S:=\mathbb{C}[y_1,\dots,y_d,p_1^{\pm 1},\dots,p_d^{\pm 1},q_1^{\pm 1},\dots,q_{d-1}^{\pm 1}]$ and consider the ideal $I\subset S$ generated by $p_i-(1+y_i)$ and $q_i-(1+y_i+y_iy_{i+1})$.
Then Corollary~\ref{cor:ui monomials} determines a map
\begin{eqnarray}\label{eq:map tilde f}
    \widetilde f: \mathbb{C}[u_1,\dots,u_{3d-1}]\to S/I
\end{eqnarray}
sending each $u_i$ to its corresponding monomial in $S$. 
The kernel of $\widetilde f$ is the prime ideal we denote by $\widetilde K$. 

\begin{exa}
    For $d=3$ we compute $\widetilde K$ in \emph{Macaulay2} and find that it is minimally generated by
    \begin{eqnarray*}
    &u_1+u_4u_5-1, \quad u_2+u_5u_7u_6-1, \quad u_3+u_6u_8-1, \quad u_4+u_1u_7-1, & \\
    & u_5+u_2u_1u_8-1, \quad u_6+u_3u_2-1, \quad u_7+u_2u_4u_8-1,  \quad u_8+u_3u_5u_7-1.&
    \end{eqnarray*}
    Up to relabelling the variables these equations coincide with the equations~\eqref{eq:equations Gamma}, as well as the $u$-equations from \cite[Problem 2.23]{Lam::LectureNotes}.
\end{exa}

In order to see that the Pellspace is irreducible we need to show that the $u$-equations in \eqref{eq:equations Gamma} are the generators of $\widetilde K$.
The following corollary, the converse of Corollary~\ref{cor:ui monomials}, is needed:

\begin{cor}\label{cor:inverse ui monimials}
    We have for $i=2,\dots,d-1$
    \begin{eqnarray*}
        &    y_1 =\tfrac{u_1u_{1+2d}}{u_{d+1}}, \quad y_i=\tfrac{u_iu_{i+2d}}{u_{i-1+2d}u_{i-1+d+1}}, \quad y_d=\tfrac{u_d}{u_{2d}u_{d-1+2d}},\quad q_1=\tfrac{1}{u_{1+d+1}u_{d+1}},&\\
        &p_1= \tfrac{1}{u_{d+1}},\quad p_i= \tfrac{1}{u_{i-1+d+1}u_{i-1+2d}}, \quad   p_d = \tfrac{1}{u_{2d}u_{3d-1}}, \quad q_i= \tfrac{1}{u_{i-1+d+1}u_{i+d+1}u_{i-1+2d}}.&
    \end{eqnarray*}
\end{cor}

\begin{prop}\label{prop: generators of K}
    The ideal $\widetilde K\subset \mathbb C[u_1,\dots,u_{3d-1}]$, kernel of the map $\widetilde f$ defined in \eqref{eq:map tilde f}, is generated by the $u$-equations for $\Delta(\Sigma_d)$ in \eqref{eq:equations Gamma}.
\end{prop}
\begin{proof}
    Consider the following maps
\begin{equation}\label{eq:maps defining U}
    \mathbb{C}[u_1^{\pm 1},\dots, u_{3d-1}^{\pm1}]\overset{f}{\longrightarrow}  S \overset{g}{\longrightarrow}  \mathbb{C}[y_1,\dots,y_d],    
\end{equation}
where $f$ is an isomorphism given by the matrix $M_d^{-1}$ and $g$ is the surjection sending $y_i$ to $y_i$, $p_i$ to $1+y_i$ and $q_i$ to $1+y_i+y_iy_{i+1}$.

Denote $K:=\ker(g\circ f)$. 
Notice that we have $K\cap \mathbb C[u_1,\dots,u_{3d-1}]=\widetilde K$. 
It suffices to show that $K$ is generated by the equations in~\eqref{eq:equations Gamma}. 
In this case, as all the equations are polynomial, the claim for $\widetilde K$ follows.
The kernel of $g$ is generated by $p_i-(1+y_i)$ and $q_j-(1+y_j+y_jy_{j+1})$ for $1\le i\le d$ and $1\le j\le d-1$. The preimages of these symbols generate $K = f^{-1}(\ker(g))$, since $f$ is an isomorphism.
    
We show how the equations~\eqref{eq:equations Gamma} are obtained as preimages of elements in $\ker(g)$. 
Conversely, keeping in mind that $u_i$ are units in $\mathbb{C}[u_1^{\pm 1},\dots,u_{3d-1}^{\pm 1}]$, the same argument will show that the ideal generated by the equations~\eqref{eq:equations Gamma} coincides with $K$.
We compute
\[
u_{d+2}u_{d+1}f^{-1}(q_1-y_1p_2-1) =1-u_1-u_{1+d+1}u_{d+1} \in K.
\]
Similarly, all equations for $u_i$ with $2\le i\le d-1$ are obtained from $f^{-1}(q_i-y_ip_{i+1}-1)$. We further observe that
\begin{align*}
         u_{2d}u_{d-1+2d}f^{-1}(p_d-1-y_d) &= 1 - u_d-u_{2d}u_{d-1+2d}\\
         \text{and} \quad u_{d+1}f^{-1}(p_1-1-y_1)&=1-u_{d+1}-u_1u_{1+2d}.
\end{align*}
The next relations take one more step -- for example $1-u_{2d+1}-u_2u_{d+1}u_{2d+2}$ is obtained from multiplying the following expression by $u_{d+1}$:
\begin{align*}
    f^{-1}(p_1-y_1-1)-u_{1+2d}u_{d+2}f^{-1}(q_1-y_1p_2-1)    - u_{d+3}u_{2d+1}f^{-1}(p_2-y_2-1).
\end{align*}
Similar expressions exist for the $u$-equations for $u_{2d+i}, \; i=2,\dots,d-1$. 
Finally, the expression $1-u_{i+d+1}-u_{i+1}u_{i+1+2d}$ equals
\begin{align*}
     u_{i-1+d+1}u_{i+d+1}u_{i-1+2d}f^{-1}\big((q_i-1-y_i-y_iy_{i+1})&-(p_i-y_i-1)\big)
\end{align*}
recovers the $u$-equations for $u_{i+d+1}$ for $1\le i\le d-1$.
\end{proof}

We are now prepared to prove Theorem~\ref{thm:main}:

\begin{proof}[Proof of Theorem~\ref{thm:main}]
The variety $\widetilde{\mathcal U}_d$ is by definition the affine closure of the variety defined by the vanishing of $K=\ker(g \circ f)$ in $(\mathbb C^*)^{3d-1}$ (or equivalently, $\widetilde{\mathcal U}_d$ is the vanishing set of $\widetilde K=\ker(\widetilde f)$ in $\mathbb C^{3d-1}$ as defined above).
By Proposition~\ref{prop: generators of K}, $K$ is generated by the equations~\eqref{eq:equations Gamma}, which are of the desired form~\eqref{eq:u-equations} by Lemma~\ref{lem:u-eqn Gamma}.
We proceed by verifying the items (i), (ii) and (iii) of Definition~\ref{Def:BinaryGeom}.

For (i) we need to show that $\widetilde{\mathcal U}_d$ is an irreducible variety of dimension $d$.   
Irreducibility follows from Proposition~\ref{prop: generators of K}.
Recall that a surjection of coordinate rings corresponds to an inclusion of the affine varieties. 
In particular, the existence of the map $g$ in \eqref{eq:maps defining U} shows that 
\[
    \mathbb{C}(\widetilde{\mathcal U}_d)=\text{Frac}( \mathbb{C}[u_1,\dots,u_{3d-1}]/K)\cong  \mathbb{C}(y_1,\dots,y_d),
\]
so ${\mathcal U}_d$ is of dimension $d$. 

For (ii) and (iii), we use Lemma~\ref{Lemma:Link} and induction on $d$. 
For $d=1$, Example~\ref{Ex:d1BinaryGeom} shows that $\widetilde{\mathcal U}_1$ is a binary geometry. 
Assume that for all $d' < d$ we have that $\widetilde U_{d'}$ is a binary geometry. 
From Lemma~\ref{Lemma:Link} and Lemma~\ref{lem:stars} for each $k \in [3d-1]$ it follows that $(\widetilde{\mathcal U}_d)_{\{k\}}$ is a product of binary geometries.
Hence $(\widetilde{\mathcal U}_d)_{\{k\}}$  is a binary geometry by Proposition~\ref{Prop:ProductBinGeom}. 
Using Lemma~\ref{Lemma:Link}(b), we conclude that the Pellspace $\widetilde{\mathcal U}_d$ is a binary geometry for $\Delta(\Sigma_d)$.
\end{proof}

\subsection{Relationship between the Pellspace and $\widetilde{\cM}_{0,n}$}\label{sec:curves}

In this section we compare the binary geometries given by the pellytope and the associahedron. 
Let $n = d+3$ and consider an $n$-gon with cyclically labelled vertices. We label the arcs of the $n$-gon by $ij$ where $1 \leq i < j-1 \leq n-2$, and say that two arcs are incompatible if they cross each other. The $u$-equations determined by the associahedron are given by
\begin{equation}\label{eqn: u-eqns for associahedron}
    u_{ij} + \prod_{kl \not \sim ij} u_{kl} = 1.
\end{equation}
The ABHY construction of the associahedron in kinematic space is of particular interest due to its connection to the positive geometry on $\cM_{0,n}$~\cite{ABYS}. 
This realisation~$\mathcal{A}_{n-3}$ of the asociahedron can be defined~\cite[(5.5)]{AHL::Stringy} as the Newton polytope of the polynomial
\begin{equation}\label{eqn: associahedron as a Newton polytope}
        G_{n-3} := \prod_{ij}(1 \, +\,  y_i\,  + \, y_iy_{i+1} \, +\,  \cdots \, +\,  y_iy_{i+1} \cdots y_{j-2}) \; \in \CC[y_1, \ldots, y_{n-3}].
\end{equation}
Cones in the normal fan $\Sigma_{\mathcal{A}_{n-3}}$ of the associahedron correspond to subdivisions of the $n$-gon -- rays correspond to arcs $ij$.
The ABHY realisation is equivalent to setting the positive orthant in $\mathbb{R}^{n-3}$ to be the maximal cone corresponding to the triangulation of the $n$-gon given by every arc centred at a single point.
Labelling this point $n-1$ gives the following dictionary between arcs on the $n$-gon and primitive generators of the rays of $\Sigma_{\mathcal{A}_{n-3}}$:
\begin{equation}\label{eqn: dictionary arcs to rays}
    e_i \longleftrightarrow i \,(n-1), \quad -e_i \longleftrightarrow (i+1) \,n \quad \text{and} \quad e_i - e_k \longleftrightarrow i \,(k+1).
\end{equation}  
We note that $\Sigma_{\mathcal{A}_{d}}$ is a refinement of $\Sigma_d$ -- the polynomial defining $\mathcal{P}_d$ clearly divides $G_d$, so the pellytope $\mathcal{P}_d$ is a Minkowski summand of $\mathcal{A}_d$.

\begin{exa}
    When $d=1$ or $2$, the pellytope and the ABHY associahedron coincide. In the case $d=3$, the normal fan to $\mathcal{A}_3$ is given by adding a single ray $\cone(e_1 - e_3)$ to $\Sigma_3$. 
\end{exa}

The binary geometry defined by the associahedron $\mathcal{A}_{n-3}$ is $\widetilde{\cM}_{0,n}$, an affine chart on the moduli space of stable curves $\overline{\cM}_{0,n}$ (cf.~\cite[Lectures 1-2]{Lam::LectureNotes}).
We can therefore use the relationship between $\mathcal{P}_d$ and $\mathcal{A}_{d}$ to give a moduli interpretation of $\widetilde{\mathcal U}_d$.
Corollary~\ref{cor:M0n and pellytope} is a direct consequence of the following Lemma:

\begin{lemma}
\label{Lemma:BlowUp}
Suppose that $f_1$ and $f_2 \in \CC[y_1, \ldots, y_d]$ are polynomials with non-vanishing constant term.
Let $U_i \subset (\CC^*)^d$ be the very affine variety given by the locus where $f_i \neq 0$.
Let $\widetilde{U}_i = \Spec \CC[\Gamma_i]$ be the affine closure, where $\Gamma_i$ denotes the semigroup of bounded characters on $U_i$.

If $f_2$ divides $f_1$, then there is a birational morphism $\pi: \widetilde{U}_1 \rightarrow \widetilde{U}_2$ which is the restriction of a toric blowup of projective toric varieties $X_1 \rightarrow X_2$.
\end{lemma}

\begin{proof}
    Let $X_i$ be the projective toric variety associated to the normal fan of the Newton polytope $P_i = \Newt f_i$.
    The polynomial $f_i$ determines a section of the very ample line bundle on $X_i$ associated to $P_i$ -- let $H_i \subset X_i$ be the zero locus of this section. 
    The affine variety $\widetilde{U}_i$ may be identified with $X_i \setminus H_i$ (cf.~\cite[Proposition 5.19]{Lam::LectureNotes}).
    If $f_2$ divides $f_1$, then the normal fan of $P_1$ is a refinement of the normal fan of $P_2$, which induces a toric blowup $\pi: X_1 \rightarrow X_2$.
    Moreover, one can check that $\pi^{-1}(H_2) \subset H_1$, so the toric morphism $\pi$ restricts to a map $\widetilde{U}_1 \rightarrow \widetilde{U}_2$.
\end{proof}

Moreover, Lemma~\ref{lem:stars} shows that the boundary of $\widetilde{\mathcal U}_d$ has a recursive structure similar to that of $\overline{\cM}_{0,n}$ -- the strata of $\widetilde{\mathcal U}_d$ are isomorphic to products of lower-dimensional Pellspaces $\widetilde{\mathcal U}_i$ for a collection of $i <d$.

Since the blowup $X_{\Sigma_{\mathcal{A}_{d}}} \rightarrow X_{\Sigma_d}$ is toric, the exceptional locus is contained in the complement of the very affine varieties $\cM_{0,n}$ and $\mathcal{U}_d$. 
This suggests we can consider $\widetilde{\mathcal U}_d$ to be an affine chart on some smaller compactification of $\cM_{0,n}$ than the space of $n$-pointed stable curves $\overline{\cM}_{0,n}$.
We give a detailed description of this compactification in the case $d=3$ in the following example.

\begin{figure}[t]
\centering
\begin{minipage}[h]{0.47\textwidth}
\centering
\includegraphics[scale=0.6]{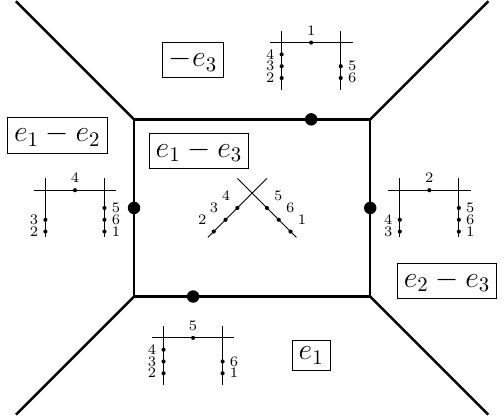}
\caption*{\small (a) $F \subset \mathcal{A}_3$ and its adjacent faces.
    }
\end{minipage}
\hfill
\begin{minipage}[h]{0.5\textwidth}
\centering
\includegraphics[scale=0.6]{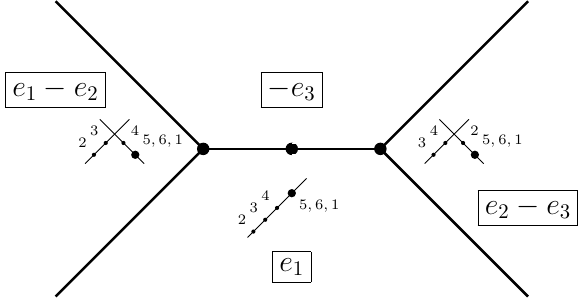}
\caption*{\small(b) $E \subset \mathcal{P}_3$ and its adjacent faces.
}
\end{minipage}

\caption{Figures (a) and (b)  above respectively show $F \subset \mathcal{A}_3$, the face of the associahedron corresponding to the diagonal $14$ on the hexagon, and $E \subset \mathcal{P}_3$, the edge of the pellytope dual to $\cone (e_1, -e_3) \in \Sigma_3$, as well as their adjacent faces.
Each of these faces is labelled with its primitive inward normal vector.
The face $F$ and its edges, as well as the edge $E$ and its endpoints, are also labelled with diagrams of the curves represented by points in the interior of the corresponding strata of $\widetilde{\mathcal{M}}_{0,6}$ and  $\widetilde{\mathcal{U}}_{3}$ respectively. }
\label{fig: facets}
\end{figure}

\begin{exa}
When $d=3$, the exceptional divisor of the toric blowup is the toric divisor $D_\rho \subset X_{\Sigma_{\mathcal{A}_3}}$ associated to the ray $\rho = \cone(e_1 - e_3)$, the only ray of $\Sigma_{\mathcal{A}_3}$ not contained in $\Sigma_3$. 
By the dictionary~(\ref{eqn: dictionary arcs to rays}), the intersection $D_\rho \cap \widetilde{\cM}_{0,6}$ is given by the equation $u_{14} =0$.

Which stable curves are contained in this stratum of $\widetilde{\cM}_{0,6}$?
The variables $u_{ij}$ are \textit{dihedral coordinates} on $\cM_{0,n}$ -- these are cross-ratios
\begin{equation}\label{eqn: cross-ratio}
    u_{ij} = \frac{(x_i - x_{j+1})(x_{i+1} - x_j)}{(x_i - x_j)(x_{i+1} - x_{j+1})},
\end{equation}
where each point on $\cM_{0,n}$ represents a choice of (a $\text{PGL}(2)$-orbit of) $n$ distinct points $p_i := [x_i:1] \in \mathbb{P}^1$ -- that is, the isomorphism class of a smooth rational curve with $n$ marked points.
The cross-ratio~(\ref{eqn: cross-ratio}) determines the image $[u_{ij}:1]$ of the point $p_{j+1}$ under the $\text{PGL}(2)$-transformation that sends $(p_i, p_j, p_{i+1})$ to $(0,1,\infty)$.
Thus the point $p_{j+1}$ collides with $p_i$ as $u_{ij} \rightarrow 0$, and $p_{j+1}$ collides with $p_j$ as $u_{ij} \rightarrow 1$.

One sees that the marked points $p_5$ and $p_1$ collide on our six-pointed stable curve as $u_{14} \rightarrow 0$.
However, the relations~(\ref{eqn: u-eqns for associahedron}) satisfied by the dihedral coordinates imply that $u_{25}, \, u_{26}, \, u_{35}, \, u_{36} \rightarrow 1$ as $u_{14} \rightarrow 0$, so the three points $p_5$, $p_6$ and $p_1$ collide as $u_{14} \rightarrow 0$.
We consider a point in $\{u_{14} =0\}$ to represent the \textit{stabilisation} of the limit of a one-parameter family in $\cM_{0,6}$ of smooth curves $C_t$ in which the marked points $p_5$, $p_6$ and $p_1$ collide as $t \rightarrow 0$.
In the stabilisation $C_0$, the three colliding points break off into a second component $C'$, and are distributed on $C'\cong \mathbb{P}^1$ according to their ratios of approach to one another.
Indeed, one can check that $\{u_{14} =0\} \cap \widetilde{\cM}_{0,6} \cong \widetilde{\cM}_{0,4} \times \widetilde{\cM}_{0,4}$ -- each fibre $\widetilde{\cM}_{0,4}$ parametrises one of the two irreducible components of $C_0$ (each considered as four-pointed curves in order to encode the position of the intersection point $q$).

The image of the stratum of $\widetilde{\cM}_{0,n}$ corresponding to a cone $\sigma \in \Sigma_{\mathcal{A}_{d}}$ is an open subset of the stratum $(\widetilde{\mathcal U}_d)_\tau$, where $\tau$ is the smallest cone in $\Sigma_d$ containing $\sigma$.
In particular we have $\pi(\{u_{14} = 0\}) \subset \{u_1 = u_6 = 0\} \cap\widetilde{\mathcal U}_3$, where $u_1$ and $u_6$ are the variables associated to the rays $\cone(e_1)$ and $\cone(-e_3) \in \Sigma_d$ by the identification~\eqref{eq: ui to ray dictionary}.
By Lemma~\ref{lem:stars}, this stratum of $\widetilde{\mathcal U}_3$ is isomorphic to $\widetilde{\mathcal U}_1 \cong \widetilde{\cM}_{0,4}$, and one can see from analysis of the adjacent strata that $\pi$ contracts the fibre of $\{u_{14} = 0\}$ which parametrises the distribution of $p_5, \, p_6, \, p_1$ and $q$ on $C'$.
We may therefore consider points on $\{u_1 = u_6 = 0\} \cap \widetilde{\mathcal U}_3$ to represent smooth (but unstable) curves on which the marked points $p_5$, $p_6$ and $p_1$ coincide and the other marked points are distinct from each other and $p_1$.


\end{exa}

\bigskip \bigskip

\noindent {\bf Acknowledgements.} 
This project started during the School on \emph{Combinatorial Algebraic Geometry from Physics} May 13-17 2024 at MPI MiS in Leipzig. We are grateful to Thomas Lam for suggesting the Problem in his lecture notes \cite{Lam::LectureNotes} and to Bernd Sturmfels for initiating the collaboration. L.B. is partially funded by the CONAHCyT CF-2023-G-106 and PAPIIT IA100724 dgapa UNAM 2024.

\begin{small}

\end{small}


\begin{thebibliography}{10}
\setlength{\itemsep}{-0.1mm}

\bibitem{ABYS} Arkani-Hamed, N. and Bai, Y. and He, S. and Yan, G.:
{\em Scattering forms and the positive geometry of kinematics, color and the worldsheet}, J. High Energy Phys.(2018), no.5, 096, 75 pp.

\bibitem{PosGeom2017}
Arkani-Hamed, N. and Bai, Y. and Lam, T.: {\em Positive geometries and canonical forms}, J. High Energy Phys.(2017), no.11, 039, 121 pp.

\bibitem{AHL::Stringy} Arkani-Hamed, N., and He, S., and Lam, T.:
{\em Stringy canonical forms}, J. High Energy Phys.(2021), no. 2, Paper No. 069, 59 pp.

\bibitem{AHLT} Arkani-Hamed, N., and He, S., and Lam, T., and Thomas, H.:
{\em Binary geometries, generalized particles and strings, and cluster algebras}, Phys. Rev. D107(2023), no.6, Paper No. 066015, 8 pp.


\bibitem{Arkani-Hamed:2013jha}
Arkani-Hamed, N. and Trnka, J.: {\em The Amplituhedron}, J. High Energy Phys.(2014), no. 10, Paper No. 030, 33 pp.
    
\bibitem{HeLi::Stringy} He, S., and Li, Z., Raman, P., and Zhang, C.:
{\em Stringy canonical forms and binary geometries from associahedra, cyclohedra and generalized permutohedra}, J. High Energy Phys.(2020), no.10, 054, 35 pp.

\bibitem{Lam::LectureNotes} Lam, T.:
{\em Moduli spaces in positive geometry}, arxiv preprint (2024), arxiv:2405.17332 [math.AG]


\bibitem{Postnikov_permutahedra}
Postnikov, A.:
{\em Permutohedra, associahedra, and beyond}, Int. Math. Res. Not. IMRN(2009), no. 6, 1026–1106.

\bibitem{Ziegler:book} Ziegler, G.~M.:
{\em Lectures on Polytopes}, Grad. Texts in Math., 152
Springer-Verlag, New York, 1995, x+370 pp.
ISBN: 0-387-94365-X


\end{thebibliography}
\end{document}